%% file: MagicLabelingQuasipolynomials.tex
\documentclass[12pt]{amsart}

\usepackage[all,arc]{xy}
\usepackage[foot]{amsaddr}
\usepackage{enumerate}
\usepackage{mathrsfs}
\usepackage{cite}   % Has a bug where Theorem 1.1 ([9]) becomes Theorem 1.1 ( [9]). Fix by writing \hspace{1sp}\cite{...}. See https://tex.stackexchange.com/a/164853
\usepackage{thmtools}
\usepackage{subfig}
\usepackage{thm-restate}
\usepackage{datetime}
\usepackage{tikz}
\usepackage{hyperref}
    \hypersetup{%
        bookmarksopen=true,   % Expand bookmarks in PDF viewer
        colorlinks=true,
        linkcolor=blue,
        citecolor=blue,
    }%
\usepackage{graphicx}
    \graphicspath{{./Figures/}}   % Set the path to graphics files. The doubled braces are necessary.
\usepackage{comment}
\usepackage[capitalize,noabbrev,nameinlink]{cleveref}   % Can cause problems if other packages load afterwards.
    \crefname{prop}{Proposition}{Propositions}   % Required to generate plural forms. See Section 8.1.3 ("Automatic \newtheorem Definitions") of cleveref documentation.
    \crefname{const}{Construction}{Constructions}
    \crefname{subsection}{Subsection}{Subsections}   % Otherwise, cleveref calls subsections ``sections''
\usepackage{float}
\usepackage{longtable}
\usepackage{xcolor}
    \definecolor{amethyst}{rgb}{0.46, 0.2, 0.88}
\usepackage{mathtools}
\usepackage{marginnote}

\calclayout

\setboolean{@twoside}{false}

\newtheorem{thm}{Theorem}[section]
\newtheorem{cor}[thm]{Corollary}
\newtheorem{prop}[thm]{Proposition}
\newtheorem{lem}[thm]{Lemma}

\theoremstyle{definition}
\newtheorem{defn}[thm]{Definition}

\newtheorem{exmp}[thm]{Example}
\newtheorem{example}[thm]{Example}

\theoremstyle{remark}
\newtheorem{rem}[thm]{Remark}

\makeatletter
\let\c@equation\c@thm
\makeatother

\numberwithin{equation}{section}   % Number equation 2 in section 3 as (3.2)

\makeatletter
\newcommand*\bigcdot{\mathpalette\bigcdot@{.5}}
\newcommand*\bigcdot@[2]{\mathbin{%
    \vcenter{\hbox{\scalebox{#2}{$\m@th#1\bullet$}}}
}}%
\makeatother

\makeatletter
\def\subsection{\@startsection{subsection}{3}%
    \z@{.5\linespacing\@plus.7\linespacing}{.1\linespacing}%
    {\bfseries}}
\makeatother

\newcommand{\C}{\mathbb{C}}

\newcommand{\R}{\mathbb{R}}
\newcommand{\Z}{\mathbb{Z}}

\DeclareMathOperator{\mprec}{mp}   % matching preclusion number
\newcommand{\Lstar}{L_{*}}

\theoremstyle{plain}
\newtheorem{theorem}[thm]{Theorem}
\newtheorem{lemma}[thm]{Lemma}
\newtheorem{corollary}[thm]{Corollary}

\newtheorem{proposition}[thm]{Proposition}

\newcommand{\deftobe}{\coloneqq}   % := for defining new notation
\newcommand{\betodef}{\eqqcolon}   % =: 
\newcommand{\divides}{\mathrel{\vert}}

\newcommand{\maps}{\colon}
\newcommand{\sst}{\,:\,}   % "such that" for set-builder notation
\newcommand*{\Q}{\mathbb{Q}}
\providecommand*{\C}{\mathbb{C}} % hyperref with unicode option uses \C for the ``combining double grave accent'' U+030F, which turns x into x̏.

\newcommand{\fd}{\Delta}   % first difference operator
\DeclareMathOperator{\lcm}{lcm}
\DeclareMathOperator{\den}{den}   % denominator of a polytope
\DeclareMathOperator{\vertx}{vert}   % vertex-set of a polytope
\DeclareMathOperator{\ehr}{ehr}   % Ehrhart quasipolynomial
\DeclareMathOperator{\mqp}{mqp}   % minimum quasiperiod

\renewcommand{\subset}{\subseteq}
\renewcommand*{\pmod}[1]{\:(\operatorname{mod} {#1})}
\renewcommand{\mod}{\pmod}

\DeclarePairedDelimiter{\braces}{\lbrace}{\rbrace}

\DeclarePairedDelimiter{\floors}{\lfloor}{\rfloor}
\DeclarePairedDelimiter{\verts}{\lvert}{\rvert}

\newcommand*{\biggparens}[1]{\biggl\lparen #1 \biggr\rparen}

\newcommand*{\card}{\verts*}
\newcommand*{\floor}{\floors*}

\newcommand*{\setof}[1]{\braces*{#1}}
\newcommand{\defing}[1]{\textup{\textbf{#1}}}   % for defining terms
\newcommand{\unisubn}{\unichar{"2099}}   % subscript n
\newcommand{\uniF}{\unichar{"1D439}}     % math F
\newcommand{\unin}{\unichar{"1D45B}}     % math n
\newcommand{\uniG}{\unichar{"1D43A}}     % math G

\title{Quasiperiods of Magic Labeling Quasipolynomials}

\author[Bayer]{Margaret Bayer\textsuperscript{*}}
\address{\textsuperscript{*} University of Kansas, Lawrence, KS, USA}
\author[Burcroff]{Amanda Burcroff\textsuperscript{\textbullet}}
\address{\textsuperscript{\textbullet} Harvard University, Cambridge, MA, USA}
\author[McAllister]{Tyrrell B. McAllister\textsuperscript{\textparagraph}}
\address{\textsuperscript{\textparagraph} University of Wyoming, Laramie, WY, USA}
\author[Pai]{Leilani Pai\textsuperscript{\S}}
\address{\textsuperscript{\S} University of Nebraska-Lincoln, Lincoln, NE, USA}

\begin{document}

\begin{abstract}
    A \defing{magic labeling} of a graph is a labeling of the edges by
    nonnegative integers such that the label sum over the edges
    incident to every vertex is the same.  This common label sum is
    known as the \defing{index}.  We count magic labelings by maximum
    edge label, rather than index, using an Ehrhart-theoretic
    approach.  In contrast to Stanley's 1973 work showing that the
    function counting magic labelings with bounded index is a
    quasipolynomial with quasiperiod $2$, we show by construction that
    the minimum quasiperiod of the quasipolynomial counting magic
    labelings with bounded maximum label can be arbitrarily large,
    even for planar bipartite graphs.  Unfortunately, this rules out a
    certain Ehrhart-theoretic approach to proving Hartsfield and
    Ringel's Antimagic Graph Conjecture.  However, we show that this
    quasipolynomial is in fact a polynomial for any bipartite graph
    with matching preclusion number at most $1$, which includes any
    bipartite graph with a leaf.
\end{abstract}

\maketitle

\section{Introduction}\label{sec: intro}

A \defing{magic labeling} of a graph is a function assigning to each
edge of the graph a nonnegative integer so that the sum of the labels
on the edges containing each vertex is the same.  This common sum at
each vertex is called the \defing{index} of the magic labeling.  The
study of magic labelings of graphs was initiated by a problem proposed
by Ji\v{r}\'{i} Sedl\'{a}\v{c}ek in 1966~\cite{MR0172259}, and the
first paper devoted to the topic was by B.~M.~Stewart~\cite{Stewart}.
Interest in this topic grew significantly after the publication of
Richard Stanley's paper, ``Linear homogeneous Diophantine equations
and magic labelings of graphs''~\cite{Sta1973}.  In that paper,
Stanley showed that the number of magic labelings with index~$k$ is a
quasipolynomial function of $k$ with minimum quasiperiod at most~$2$.

In this paper, we count the number of magic labelings of a graph by
the maximum label used, rather than by the index.  The function that
counts the number of magic labelings with maximum label at most $k$ is
again a quasipolynomial.  Our main result (\autoref{thm: main theorem}
below) is that, in contrast to Stanley's result, the minimum
quasiperiod of this quasipolynomial is unbounded.  We show this by
constructing, for each positive integer $n$, a planar bipartite graph
for which the minimum quasiperiod is~$n$.

This result is motivated by work of Matthias Beck and Maryam Farahmand
in \cite{BecFar2017} on \emph{anti}magic labelings, which are
labelings of the edges of a graph by distinct labels in
$\{1,2,\dots,|E|\}$ such that the sums of the labels at each vertex
are \emph{distinct}.  A famous open problem from 1990 posed in
\cite{pearls} is whether all connected graphs (except for~$K_2$) have
an antimagic labeling.  See \cite[Chapter~6]{Gallian} for a
comprehensive summary of progress on this problem.  Beck and Farahmand
pursued a strategy for proving a weakened form of this conjecture,
namely that for some fixed $s \geq 1$, every connected graph (except
for~$K_{2}$) has a labeling using only labels in $\setof{1, \dotsc,
s\card{E}}$ (allowing repeated labels) where the sums of the labels at
each vertex are distinct.  As shown in~\cite{BecFar2017}, this claim
would follow if it were known that the function that counts the number
of magic labelings with maximum label~$k$ has minimum quasiperiod at
most~$s$ for every graph.  Unfortunately, the present paper shows that
no such bound on the minimum quasiperiod exists for general graphs, so
the approach explored in~\cite{BecFar2017} cannot succeed without
significant modification.

While this minimum quasiperiod is unbounded in general, we show that
the minimum quasiperiod is bounded for certain classes of graphs (see
\autoref{sec: small quasiperiod}).  In particular, we show that, if a
graph contains an edge that attains the maximum label in every
index-$2$ magic labeling, then this quasipolynomial has minimum
quasiperiod at most $2$.  Furthermore, if the graph is additionally
bipartite, then this quasipolynomial is in fact a polynomial.  This
suggests that Beck and Farahmand's approach may be adapted for special
families of graphs.  However, the classes of graphs that we consider
are not closed under certain operations that they use in
\cite{BecFar2017}.

\section{Preliminaries}\label{sec: prelims}

Throughout this paper we let $G = (V,E)$ be an undirected graph
without multiple edges (but possibly with loops).

An (\defing{edge}) \defing{labeling} of $G$ is a function $E \to
\Z_{\ge 0}$.  We view the labelings of~$G$ as the points of the
integer lattice $\Z^E$ that are in the positive orthant
$\smash{\R_{\ge 0}^E}$ of the vector space $\R^E$.  For each $L \in
\R^E$, write
\begin{equation*}
    s_{L}(v) \deftobe \sum_{e\in E,\, e \ni v} L(e)
\end{equation*}
for the sum of the labels of the edges incident to $v$.  A labeling
$L$ is \defing{magic} if the value of $s_{L}(v)$ is the same for each
vertex $v$ of $G$.  Thus, the magic labelings are the lattice points
in the polyhedral cone $C_{G} \subset \R^E$ defined by
\begin{equation*}
    C_{G}
    \deftobe
    \setof{L \in \R_{\ge 0}^E \sst \text{$s_{L}(v) = s_{L}(w)$ for
    all $v, w \in V$}}.
\end{equation*}
See \cite[Chapter~5]{Gallian} for a survey of results concerning magic
labelings and related notions.

The primary object of study in this paper is the rational polytope
$P_{G} \subset \R^E$ defined by
\begin{equation*}
    P_{G} \deftobe C_{G} \cap [0,1]^{E}.
\end{equation*}
For $k \in \Z_{\ge 0}$, a labeling $L$ is a \defing{$k$-labeling} if
$L(e) \le k$ for all edges $e \in E$.  Thus, the magic $k$-labelings
of $G$ are precisely the integer-lattice points in the
$k$\textsuperscript{th} dilate ${k P_{G} \deftobe \setof{k L \sst L
\in P_{G}}}$ of $P_{G}$.  We are in particular interested in the
function
\begin{equation*}
    M_{G}(k) \deftobe \card{k P_{G} \cap \Z^{E}}
\end{equation*}
that counts the number of magic $k$-labelings of $G$.  

A better-studied counting function in the context of magic labelings
is the function $S_{G}(k)$ that counts the magic labelings of $G$ with
index exactly~$k$, where the \defing{index} of a magic labeling~$L$ of
$G$ is the common value of $s_{L}(v)$ for all $v \in V$.  This
function corresponds to the polytope
\begin{equation*}
    Q_{G} 
    \deftobe%
    \setof{%
        L \in \R_{\ge 0}^E
        \sst
        \text{$s_{L}(v) = 1$ for all $v \in V$}
    },%
\end{equation*}
since the number of magic labelings of $G$ with index $k$ is
\begin{equation*}
    S_{G}(k) = \card{k Q_{G} \cap \Z^E}.
\end{equation*}
Note that, since labelings are nonnegative, a magic labeling with
index~$k$ is in particular a magic $k$-labeling.  The corresponding
statement regarding the polytopes is that $Q_G \subset P_G$.  However,
the dimension of $Q_G$ is always strictly less than that of $P_G$.

From the point of view of Ehrhart theory, the definitions of $M_{G}$
and of $S_{G}$ immediately imply that they are the \emph{Ehrhart
quasipolynomials} of the polytopes~$P_{G}$ and $Q_{G}$, respectively.
We briefly explain this connection, but we refer the reader to
\cite{BecRob2007} for a thorough introduction to Ehrhart theory,
including the properties of Ehrhart quasipolynomials stated here.

A function $F \maps \Z \to \C$ (or $F \maps \Z_{\ge 0} \to \C$) is a
\defing{quasipolynomial} of \defing{degree}~$d$ if there exist an
integer $s \in \Z_{\ge 1}$ and polynomials $\phi_{1}, \dotsc, \phi_{s}
\in \C[x]$, called the \defing{constituents} of $F$, such that $d =
\max\setof{\deg(\phi_{1}), \dotsc, \deg(\phi_{s})}$ and ${F(t) =
\phi_{r}(t)}$ whenever $t \equiv r \mod{s}$.  Such a positive
integer~$s$ is a \defing{quasiperiod} of~$F$.  The quasiperiods of $F$
are precisely the positive integer multiples of the \emph{minimum}
quasiperiod of $F$, which we denote $\mqp(F)$.  Alternatively, $F$ is
a quasipolynomial of degree~$d$ if and only if $F(t) = \sum_{i=0}^{d}
c_{i}(t) \, t^{i}$ for some sequence $c_{0}, \dotsc, c_{d}$ of
periodic \defing{coefficient functions} $\Z \to \C$ with $c_{d}$ not
identically zero.  Writing $s_{i}$ for the minimum period of $c_{i}$
for $0 \le i \le d$, we then have that $\mqp(F) = \lcm\setof{s_{0},
\dotsc, s_{d}}$.

Now let $P \subset \R^{N}$ be a $d$-dimensional \defing{rational}
polytope, meaning that $\vertx(P) \subset \Q^{N}$, where $\vertx(P)$
denotes the set of vertices of~$P$.  By a celebrated theorem of
Eug{\`e}ne Ehrhart, the function $\ehr_{P}(t) \deftobe \card{tP \cap
\Z^{N}}$ for $t \in\Z_{\ge 1}$ is a degree-$d$ quasipolynomial called
the \defing{Ehrhart quasipolynomial} of~$P$.  Moreover, the
\emph{minimum} quasi\-period of $\ehr_{P}$ divides the
\defing{denominator} of $P$, which is defined to be $\den(P) \deftobe
\min\setof{t \in \Z_{\ge 1} \sst \vertx(tP) \subset \Z^{N}}$.  Thus,
an upper bound on the denominator of $P$ is also an upper bound on
$\mqp(\ehr_P)$.

As discussed in the introduction, Beck and Farahmand showed in
\cite{BecFar2017} that a proof of an upper bound on $\mqp(M_{G})$
independent of $G$ would suffice to prove a weakened version of an
open problem regarding antimagic graph labelings.  However, we find
below that no such upper bound on $\mqp(M_{G})$ exists.

In order to compute the denominator and minimum quasiperiod of a
rational polytope $P \subset \R^N$, we study a related semigroup in
$\R^{N+1}$.  The \defing{semigroup of~$P$}, denoted by $\Phi(P)$, has
elements $(L, k)$ where $k$ is a nonnegative integer and $L$ is a
lattice point in $kP$.  That is, the semigroup $\Phi(P)$ consists of
the integer points in the \defing{homogenized cone over $P$}, which is
the cone generated by $P \times \setof{1}$ in $\R^{N+1}$.  The binary
operation in the semigroup is entry-wise addition.
% , and an element is \defing{nontrivial} if it has a nonzero entry.

\begin{defn}
    A nonzero element $a$ of an additive semigroup~$\Phi$ is
    \defing{completely fundamental} if, for all $b, c \in \Phi$, if $b
    + c = ma$ for some positive integer $m$, then $b$ and $c$ are both
    nonnegative integer multiples of $a$.
\end{defn}

We remark that, when $\Phi = \Phi(P)$ is the semigroup of a rational
polytope~$P$, then the completely fundamental elements of $\Phi$ are
precisely the points $(d_v v, d_v)$ where $v$ is a vertex of $P$ and
$d_v \deftobe \den(v)$.  Thus, the denominator of~$P$ can be expressed
in terms of the completely fundamental elements of $\Phi(P)$.

\begin{cor}\label{cor: quasiperiod ub}
    The denominator of a polytope $P$ is equal to the least common
    multiple of the final coordinates of the completely fundamental
    elements of $\Phi(P)$.
\end{cor}

When $P \in \setof{P_G,\, Q_G}$, we say that a magic labeling $L$ is a
\defing{completely fundamental (magic) labeling of $\Phi(P)$} if
$(L,k)$ is a completely fundamental element of $\Phi(P)$ for some $k
\in \Z_{\geq 0}$.  An important subtlety is that which labelings are
``completely fundamental" depends upon whether one is considering the
polytope $P_G$ of all magic labelings with labels $\le k$ or the
polytope $Q_G$ of magic labelings of index exactly $k$.

\begin{example}\label{exmp: polytopes and cones}
    Let~$G$ be the graph with one node $v$ and two loops at~$v$.  Then
    we can identify~$\R^{E}$ with $\R^{2}$.  Under this
    identification, $P_{G}$ is the unit square $[0,1]^{2}$, and
    $Q_{G}$ is the line segment with endpoints $(1,0)$ and $(0,1)$.
    In \autoref{fig: cone example}, the semigroup $\Phi(P_{G})$ is the
    set of integer-lattice points (not shown) in the blue cone, and
    $\Phi(Q_{G})$ is the set of such points in the red cone.  The
    completely fundamental elements of $\Phi(P_{G})$ are $(0,0,1)$,
    $(0,1,1)$, $(1,0,1)$, and $(1,1,1)$, corresponding to the four
    vertices of $P_G$.  The completely fundamental elements of
    $\Phi(Q_{G})$ are $(0,1,1)$ and $(1,0,1)$, corresponding to the
    two vertices of $Q_G$.
\end{example}

\begin{figure}
    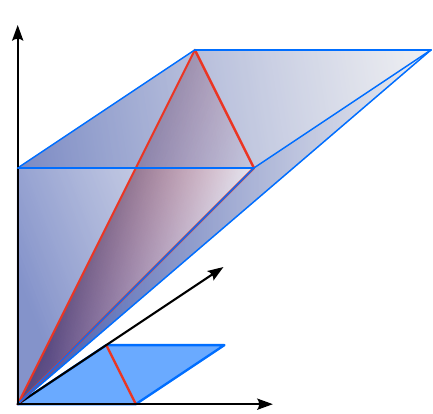
    \caption{%
        The polytopes $P_G$ and $Q_G$ from \autoref{exmp: polytopes
        and cones} are shown in the $xy$-plane in blue and red,
        respectively.  A portion of the homogenized cone over each
        polytope is shown in the corresponding color.
    }%
    \label{fig: cone example}
\end{figure}

In 1973, Stanley proved a strong bound on the denominator of the
polytope~$Q_{G}$, which in turn implies the same bound on the minimum
quasiperiod of the quasipolynomial $S_{G}(k)$ that counts the
index-$k$ magic labelings of $G$.  Stanley stated his result in terms
of the ``completely fundamental magic labelings of $G$'', which in our
nomenclature meant the completely fundamental labeling of $\Phi(Q_G)$
specifically.

\begin{lem}[\hspace{1sp}{\cite[Proposition~2.7]{Sta1973}}]%
\label{lem: stanley cf index}%
    For a finite graph $G$, every completely fundamental magic
    labeling of $\Phi(Q_G)$ has index $1$ or $2$.  If $G$ is
    additionally bipartite, then every completely fundamental magic
    labeling of $\Phi(Q_G)$ has index~$1$.
\end{lem}

Applying \autoref{cor: quasiperiod ub} yields the following uniform
bounds.

\begin{proposition}[\hspace{1sp}{\cite[Corollary~2.8]{Sta1973}}]
    For all graphs $G$, the denominator of the polytope $Q_{G}$ is at
    most~$2$.  In particular, the minimum quasiperiod of $S_{G}(k)$ is
    at most~$2$.
\end{proposition}

Our main theorem is that \emph{no such bounds exist} on either the
denominator of the polytope $P_{G}$ or on the minimum quasiperiod of
the quasipolynomial $M_{G}(k)$ that counts the magic
$k$-labelings.%
\footnote{%
    A claim to the contrary appears as Theorem~4 of~\cite{BecFar2017}.
    However, the authors of~\cite{BecFar2017} report in private
    correspondence that the proof of their Theorem~4 contained an
    error and that an erratum is forthcoming.
} %

\begin{theorem}\label{thm: main theorem}
    There exist graphs $G$ for which the minimum quasiperiod
    of~$M_{G}$ is arbitrarily large.  In particular, for each $n \in
    \Z_{\ge 2}$, there exists a graph $G_{n}$ \textup{(}on $2n+2$
    vertices and $3n$ edges\textup{)} such that $P_{G_{n}}$ has a
    vertex with denom\-inator~$n-1$, and the minimum quasiperiod of
    the quasipolynomial $M_{G_{n}}$ is~$n-1$.
\end{theorem}

The proof of this theorem appears at the end of \cref{sec: Gn
quasipolynomial}.  The outline of the argument is as follows.  The
construction of $G_{n}$ is at the beginning of \autoref{sec: unbounded
max label}, and the fact that $P_{G_{n}}$ has a vertex with
denominator $n-1$ is a direct consequence of \autoref{thm: vertices of
PGn}.  Now, since the denominator is only an upper bound on the
minimum quasi\-period of $M_G$, the strategy explored in
\cite{BecFar2017} might still be salvaged if $P_G$ exhibited so-called
``period collapse" \cite{McA2021}.  However, in \cref{prop: Mn to
Fn-1,cor: quasiperiod of Fn}, we establish that the Ehrhart
quasipolynomial $M_{G_n}$ of $P_{G_n}$ has ``full period".  That is,
the quasiperiod of the quasipolynomial attains the upper bound in
\autoref{cor: quasiperiod ub}.  In particular, the minimum quasiperiod
of $M_{G_{n}}$ is $n-1$, as claimed in \cref{thm: main theorem}.

% Thus the quasiperiod of the Ehrhart quasipolynomial
% $M_{G_n}(k)$ becomes arbitrarily large as $n$ increases.

\begin{exmp}
    Consider $4$ copies of the path on $3$ vertices, where the ends of
    the paths are labeled by $x_i$ and $y_i$ for $1 \leq i \leq 4$.
    Let $G_4$ be the graph on $10$ vertices obtained by identifying
    all the $x_i$ vertices as a single vertex~$x$ and all the $y_i$
    vertices as a single vertex $y$.  Then $M_{G_4}(k)$ is the
    following quasipolynomial with minimum quasiperiod $3$:
    \begin{equation*}
        M_{G_4}(k) 
        = 
        \begin{dcases*}
            \tfrac{1}{18}k^4 + \tfrac{4}{9}k^3 + \tfrac{25}{18}k^2 + 2k + 1 
                & if $k \equiv 0 \mod{3}$, \\
            \tfrac{1}{18}k^4 + \tfrac{4}{9}k^3 + \tfrac{25}{18}k^2 + 2k + \tfrac{10}{9}
                & if $k \equiv 1 \mod{3}$, \\
            \tfrac{1}{18}k^4 + \tfrac{4}{9}k^3 + \tfrac{25}{18}k^2 + 2k + 1 
                & if $k \equiv 2 \mod{3}$.
        \end{dcases*}
    \end{equation*}
    The fact that only the degree-$0$ coefficient varies is explained
    in the remark following the proof of \cref{cor: quasiperiod of Fn}
    below.
\end{exmp}

The contrast between our results and Stanley's results is rooted in
the difference between $\Phi(P_G)$ and $\Phi(Q_G)$.  It follows from
\autoref{thm: main theorem} and \autoref{cor: quasiperiod ub} that the
analogue of \autoref{lem: stanley cf index} is false if we replace
$\Phi(Q_G)$ with $\Phi(P_G)$, and moreover no uniform bound can be
placed on the index of completely fundamental magic labelings of
$\Phi(P_G)$.

However, an analogue of \autoref{lem: stanley cf index} does hold for
certain types of graphs.  We show in \autoref{sec: small quasiperiod}
that if $G$ has an edge that attains the maximum label in every
index-$2$ magic labeling, then the completely fundamental magic
labelings of $\Phi(P_G)$ have index at most $2$.  Moreover, if $G$ is
additionally bipartite, then the completely fundamental magic
labelings of $\Phi(P_G)$ have index at most $1$.  In particular, if
$G$ is any bipartite graph with a leaf, it follows that $M_k(G)$ is a
polynomial.

\section{%
    Unbounded completely fundamental labelings
}\label{sec: unbounded max label}%

In this section, we construct a family $\setof{G_n}_{n \ge 2}$ of
graphs for which there are completely fundamental magic labelings of
$\Phi(P_{G_n})$ with arbitrarily large maximum label.  This means that
there is no uniform upper bound on the denominator of $P_{G}$.
\begin{defn}\label{defn: G_n}
    For each integer $n \geq 2$, let $G_n$ be the graph on the vertex
    set $\setof{a_1, \dotsc, a_n, b_1, \dotsc, b_n, x, y}$ with the
    following $3n$ edges:
    \begin{itemize}
        \item
        an edge from $a_i$ to $b_i$ for each $i \in \setof{1, \dotsc,
        n}$,
        
        \item
        an edge from $x$ to $a_i$ for each $i \in \setof{1, \dotsc,
        n}$, and
        
        \item
        an edge from $y$ to $b_i$ for each $i \in \setof{1, \dotsc,
        n}$.
    \end{itemize}
\end{defn}

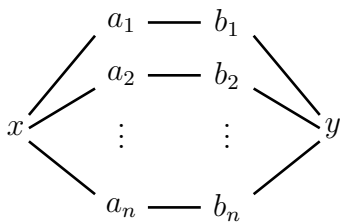
\begin{figure}%[H]
    \centering
    \begin{tikzpicture}[x=0.5pt, y=0.5pt]
        \node at (290,490) [opacity=1] {\textcolor[RGB]{0,0,0}{$a_1$}};
        \node at (290,450) [opacity=1] {\textcolor[RGB]{0,0,0}{$a_2$}};
        \node at (370,410) [opacity=1] {\textcolor[RGB]{0,0,0}{$\vdots$}};
        \node at (290,350) [opacity=1] {\textcolor[RGB]{0,0,0}{$a_n$}};
        \node at (370,490) [opacity=1] {\textcolor[RGB]{0,0,0}{$b_1$}};
        \node at (370,450) [opacity=1] {\textcolor[RGB]{0,0,0}{$b_2$}};
        \node at (370,350) [opacity=1] {\textcolor[RGB]{0,0,0}{$b_n$}};
        \node at (210,410) [opacity=1] {\textcolor[RGB]{0,0,0}{$x$}};
        \node at (450,410) [opacity=1] {\textcolor[RGB]{0,0,0}{$y$}};
        \node at (290,410) [opacity=1] {\textcolor[RGB]{0,0,0}{$\vdots$}};
        \draw[solid, draw={rgb,255:red,0;green,0;blue,0}, draw opacity=1, line width=1, ] (220,420) -- (270,480);
        \draw[solid, draw={rgb,255:red,0;green,0;blue,0}, draw opacity=1, line width=1, ] (220,410) -- (270,440);
        \draw[solid, draw={rgb,255:red,0;green,0;blue,0}, draw opacity=1, line width=1, ] (220,400) -- (270,360);
        \draw[solid, draw={rgb,255:red,0;green,0;blue,0}, draw opacity=1, line width=1, ] (440,420) -- (390,480);
        \draw[solid, draw={rgb,255:red,0;green,0;blue,0}, draw opacity=1, line width=1, ] (440,410) -- (390,440);
        \draw[solid, draw={rgb,255:red,0;green,0;blue,0}, draw opacity=1, line width=1, ] (440,400) -- (390,360);
        \draw[solid, draw={rgb,255:red,0;green,0;blue,0}, draw opacity=1, line width=1, ] (310,490) -- (350,490);
        \draw[solid, draw={rgb,255:red,0;green,0;blue,0}, draw opacity=1, line width=1, ] (310,450) -- (350,450);
        \draw[solid, draw={rgb,255:red,0;green,0;blue,0}, draw opacity=1, line width=1, ] (310,350) -- (350,350);
    \end{tikzpicture}
    \caption{The construction of the graph $G_n$.}
\end{figure}

We will construct a completely fundamental magic labeling of
$\Phi(P_{G_n})$ with maximum edge label $n-1$.

\begin{defn}\label{defn: L labeling}
    Let $\Lstar$ be the labeling on the edges of $G_n$ where edges
    from~$a_i$ to $b_i$ have label $n-1$ and all other edges of label
    $1$.
\end{defn}

It is straightforward to check that $\Lstar$ is a magic labeling of
index $n$, but it remains to show that this is a completely
fundamental labeling of $\Phi(P_{G_n})$.  In order to show this, we
first consider the perfect matchings on $G_n$.  A \defing{perfect
matching} in $G$ is a subset $J$ of the edges such that each vertex in
$G$ is incident to exactly one edge of $J$.  Note that a magic
labeling of $G$ of index $1$ can be identified with a perfect matching
of $G$ by taking the set of edges of $G$ with label~$1$.

\begin{prop}\label{prop: Gn perfect matchings}
    For $1 \leq i \leq n$, let $L_i$ be the perfect matching on $G_n$
    formed by taking the edge from $x$ to $a_i$, from $y$ to $b_i$,
    and all edges from $a_j$ to $b_j$ for $j \neq i$.  All perfect
    matchings on $G_n$ are of this form.
\end{prop}

Let $\max(L)$ denote the maximum label appearing in a labeling $L$,
and write~$\vec{0}$ for the trivial magic labeling under which every
edge is labeled~$0$.

\begin{lem}\label{lem: G_n magic labeling generators}
    Any element of $\Phi(P_{G_n})$ can be written as a nonnegative
    integer combination of the elements $(L_1,1)$, $(L_2,1)$, \dots,
    $(L_n,1)$, $(\vec{0},1)$, and $(\Lstar,n-1)$.
\end{lem}

\begin{proof}
    Fix $(L,k)$ in $\Phi(P_{G_n})$.  By subtracting off copies of
    $(\vec 0,1)$, we can assume $k=\max(L)$.  For $1\le i\le n$, let
    $u_i=L(xa_i)$, i.e., the label of the edge between $x$ and $a_i$
    in $L$.  The index of $L$ is then $\sum_{i=1}^n u_i$.  In order to
    have the correct sum at $a_j$, we must have
    $L(a_jb_j)=(\sum_{i=1}^n u_i)-u_j$.  Similarly, $L(b_jy)=u_j$ in
    order to have the correct sum at $b_j$.
    %-label of the edge between $b_j$ and $y$ must be $x_j$.  
    The maximum label in $L$ is then
    \begin{equation*}
        \max(L)
        =
        \biggparens{\sum_{i=1}^n u_i} - \min_{1 \leq i \leq n} u_i.
    \end{equation*}
    Let $m$ be such that $\displaystyle u_m = \min_{1 \leq i \leq n}
    u_i$.  It is then straightforward to check that
    \begin{equation*}
        (L,\max(L)) 
        =
        \biggparens{
            \sum_{i=1}^n 
            (u_i - u_m) \cdot (L_i,1) 
        }
        +
        u_m \cdot (\Lstar, n-1).
    \end{equation*}
    Thus, we can conclude that $(L, \max(L))$ can be decomposed as
    desired.
\end{proof}

Note that the magic labeling $\Lstar$ is equal to the sum of the magic
labelings~$L_i$ for $1 \leq i \leq n$.  However, if we additionally
choose an appropriate bound for the maximum label, this no longer
holds.  That is, the element $(\Lstar,n-1)$ is not a sum of the
elements $(L_i,1)$ in $\Phi(P_{G_n})$, which is a consequence of the
following result.

\begin{thm}\label{thm: vertices of PGn}
    The completely fundamental elements of $\Phi(P_{G_n})$ are
    $(L_1,1)$, $(L_2,1)$, \dots, $(L_n,1)$, $(\vec 0,1)$, and
    $(\Lstar,n-1)$.
\end{thm}

\begin{proof}
    By \autoref{lem: G_n magic labeling generators}, it is enough to
    show that no multiple of one of these elements can be written as a
    positive integer combination of the others.

    This clearly holds for $(\vec 0,1)$, since all other elements have
    a positive edge label.  This also holds for $(\Lstar,n-1)$ since
    $(\vec 0,1)$ and the $(L_i,1)$ satisfy the property that the last
    entry is at least the index of the labeling, and this property is
    preserved under sums.  Lastly, the theorem statement holds for
    each element $(L_i,1)$, as the only other element with label $0$
    on the edge between~$a_i$ and $b_i$ is $(\vec 0,1)$ and no
    multiple of $(\vec 0,1)$ is equal to a multiple of $(L_i,1)$.
\end{proof}

\begin{rem}
    Each graph $G_n$ can be generalized to a family of graphs
    $G_{n,p}$ for $p \ge 1$ as follows, while retaining the same
    Ehrhart quasipolynomial to count the magic labelings.

    Let $G_{n,p}$ denote the graph obtained by taking two vertices $x$
    and $y$ and connecting them by $n$ distinct paths of length
    $2p+1$.  Since this graph is bipartite, any magic labeling can be
    decomposed as a sum of perfect matchings.  Moreover, a perfect
    matching of $G_{n,p}$ is determined by a choice of edge incident
    to $x$, as is the case for $G_n$.  This induces a bijection
    between perfect matchings on $G_n$ and perfect matchings on
    $G_{n,p}$, and it is straightforward to show that this bijection
    can be extended additively to magic labelings.  Thus, we have
    $M_{G_n}(k) = M_{G_{n,p}}(k)$.
\end{rem}

\section{%
    The Ehrhart quasipolynomial of
    \texorpdfstring{$P_{G_n}$}{\uniG\unisubn}
}\label{sec: Gn quasipolynomial}%

This section is focused on studying the Ehrhart quasipolynomial of
$P_{G_n}$, where $G_n$ is the graph constructed in \autoref{sec:
unbounded max label}.  In \autoref{subsec: computing ehrhart
quasipolynomial}, we give an explicit formula for this Ehrhart
quasipolynomial as a sum of certain binomial coefficients.
Ultimately, we are interested in finding the minimum quasiperiod of
this quasipolynomial in order to prove the main result (\autoref{thm:
main theorem}).  Using tools developed in \autoref{subsec: tools for
bounding}, we precisely compute the minimum quasiperiod in
\autoref{subsec: minimum quasiperiod}.

\subsection{%
    Computing the Ehrhart Quasipolynomial of
    \texorpdfstring{$P_{G_n}$}{\uniG\unisubn}
}\label{subsec: computing ehrhart quasipolynomial}%

Let $M_{n}(k) \deftobe M_{G_n}(k)$ be the number of integer points in
$kP_{G_n}$, where~$P_{G_n}$ is the polytope of magic labelings of
$G_n$ with maximum label at most $1$.  That is, $M_n$ is the Ehrhart
quasipolynomial of $P_{G_n}$.  Equivalently, $M_{n}(k)$ is the number
of integral magic labelings of $G_n$ with maximum label at most $k$.
For all $n \ge 1$, define the function $F_n \maps \Z_{\ge 0} \to \C$
by
\begin{equation*}
    F_n(k)
    \deftobe
    \sum_{\substack{j \in [0,k]_{\Z} \sst \\ j \equiv k \mod{n}}}
    \binom{j}{n}.
\end{equation*}

\begin{proposition}\label{prop: Mn to Fn-1}
    For $n \geq 2$ and nonnegative $k$, we have
    $$M_{n}(k) = \binom{k + n}{n} + F_{n-1}(k)\,.$$
\end{proposition}

\begin{proof}
%     Applying \autoref{lem: G_n magic labeling generators}, $M_n(k)$
%     counts the number of sums of the magic labelings $L_1, \dotsc,
%     L_n$ (defined in \autoref{prop: Gn perfect matchings}) with
%     maximum label at most~$k$.  Let $u_i$ denote the number of
%     instances of $L_i$ in such a sum.  Then the maximum label of the
%     sum of the magic labelings is $\sum_{i=1}^n u_i - \min_{1 \leq j
%     \leq n} u_j$.  
    As in the proof of \autoref{lem: G_n magic labeling generators},
    each magic labeling $L$ of $G_{n}$ is determined by the labels
    $u_{j}$ that $L$ assigns to the edges $xa_{j}$ for $1 \le j \le
    n$.  Moreover, the labeling $L$ is a $k$-labeling if and only if
    the maximum label $\sum_{j=1}^n u_j - \min_{1 \leq \ell \leq n}
    u_\ell$ is at most $k$.  Thus, we can directly calculate $M_n(k)$
    as follows:
    \begingroup%
    \allowdisplaybreaks%
    \begin{align*}
        M_n(k) 
        &=%
        \#\setof{(u_1,\dotsc, u_n) \in \Z_{\ge 0}^n \sst
        \text{$\sum_{j=1}^n u_j - u_\ell \le k$ for all $\ell$} }
        \\
        &=%
        \#\setof{(u_1,\dotsc, u_n) \in \Z_{\ge 0}^n \sst
        \text{$\sum_{j=1}^n u_j \le k$} } \\
        &\qquad+%
        \sum_{i \ge 1}\#\setof{(u_1,\dotsc, u_n) \in \Z_{\ge 0}^n
        \sst \text{$\sum_{j=1}^n u_j = k + i$ and $u_\ell \ge i$
        for all $\ell$} } \\
        &=%
        \binom{k+n}{n}
        +%
        \sum_{i \ge 1}
        \binom{k - (i-1)(n-1)}{n-1} \\
        &=%
        \binom{k+n}{n}
        +
        \sum_{i \ge 0}
        \binom{k - i(n-1)}{n-1}. \qedhere
    \end{align*}
    \endgroup%
\end{proof}

\subsection{Minimum quasiperiods and finite differences}%
\label{subsec: tools for bounding}

In this subsection, we review a result of Sam and Woods
\cite{SamWoo2010} and observe that, as a direct consequence of their
result, the minimum quasiperiod of a quasipolynomial~$F$ equals the
minimum quasiperiod of the \emph{first difference} of~$F$.  We will
then use this fact in the next subsection to show that the minimum
quasiperiod of $F_n$ is $n$.

The \defing{difference operator} $F \mapsto \fd F$ is defined as
follows.  For any complex-valued function $F$ defined on $\Z$, or on
any interval $[a,\infty) \subset \Z$, the \defing{first difference} of
$F$ is the function $\fd F$ defined on the same domain by
\begin{equation*}
    \fd F(t) \deftobe F(t+1) - F(t).
\end{equation*}
For $i \in \Z_{\ge 2}$, the \defing{$i$\textsuperscript{th}
difference} of $F$ is defined by ${\fd^{i} F \deftobe \fd (\fd^{i-1}
F)}$.  The operator $\fd$ satisfies an analogue of the fundamental
theorem of calculus: If $f$ and $F$ are functions defined on $[a,
\infty)$, then $\fd \sum_{x = a}^{t-1} f(x) = f(t)$ and
${\sum_{x=a}^{t-1} \fd F\,(x) = F(t) - F(a)}$ for all $t > a$.  This
operator thus gives rise to a rich ``calculus of finite differences''
\cite{GKP1994}.  Sam and Woods use this calculus in \cite{SamWoo2010}
to give elementary proofs of several foundational results in Ehrhart
theory.
   
We now show that the difference operator preserves the minimum 
quasi\-period of quasipolynomials:
\begin{equation}\label{eq: difference has same or smaller min period}
%     \mqp(\fd F) \le \mqp(F).
    \mqp(\fd F) = \mqp(F).
\end{equation}
The easy half of this equation is the inequality $\mqp(\fd F) \le
\mqp(F)$.  For, let ${\phi_{1}, \dotsc, \phi_{s} \in \C[x]}$ be the
constituents of $F$.  Then, for each $r \in [s]$ and $t \equiv r
\mod{s}$, we have that $\fd F(t) = \phi_{r+1}(t+1) - \phi_{r}(t)$
(indices modulo $s$), which is a polynomial function of~$t$.  Thus,
$\fd F$ is a quasipolynomial, and~$s$ is a quasiperiod $\fd F$.

To complete the proof of Equation~\eqref{eq: difference has same or
smaller min period}, it remains to prove that $\mqp(F) \le \mqp(\fd
F)$.  This inequality follows from
\cite[Lemma~2.1]{SamWoo2010}, a special case of which%
\footnote{%
    The authors of \cite{SamWoo2010} consider the more general case in
    which the upper bound of summation in the definition of $F(t)$ is
    itself a quasipolynomial function of $t$ of the form $t \mapsto
    \floor{at/b}$ for some $a, b \in \Z$.
} %
is the following.

\begin{theorem}[{\hspace{1sp}\cite[Lemma~2.1]{SamWoo2010}}]
\label{thm: sam and woods lemma}% 
    Let $f(t) = \sum_{i=0}^{d} c_{i}(t) \, t^{i}$ be a
    quasipolynomial, where $c_{i}$ is a periodic function with minimum
    period $s_{i}$ for $0 \le i \le d$.  Let $F(t) \deftobe \sum_{x =
    0}^{t-1} f(x)$ for $t \ge 1$.  Then $F(t) = \sum_{i=0}^{d+1}
    C_{i}(t) \, t^{i}$ for some periodic functions $C_{0}, C_{1},
    \dotsc, C_{d+1}$ such that the minimum period of $C_{i}$ divides
    $\lcm\setof{s_{i}, s_{i+1}, \dotsc, s_{d}}$ for $0 \le i \le d$,
    and $C_{d+1}$ is constant.%
    \footnote{%
        It may happen that $C_{d+1} = 0$, as for example when $f(t)
        \deftobe (-1)^{t}$.
%          , so that
%          $F(t) = (1 + (-1)^{t})/2$
        In general, $C_{d+1} = 0$ if and only if $\sum_{x=1}^{s_{d}}
        c_{d} (x) = 0$.  Thus, the ``anti-difference'' operator does
        not increase the degree of every quasipolynomial.
    }%
\end{theorem}
   
Equation~\eqref{eq: difference has same or smaller min period} is now
a straightforward corollary.
   
\begin{corollary}\label{cor: difference has same min period}
    Let $F \maps \Z \to \C$ be a quasipolynomial.  Then $\fd F$ is
    also a quasipolynomial, and $\mqp(\fd F) = \mqp(F).$
\end{corollary}
   
\begin{proof}
    From the argument immediately following Equation~\eqref{eq:
    difference has same or smaller min period}, it remains only to
    prove that $\mqp(F) \le \mqp(\fd F)$.  Let $$H(t) \deftobe F(t) -
    F(0) = \sum_{x = 0}^{t-1} \fd F(x).$$ Write $H(t) \betodef
    \sum_{i=0}^{d+1} C_{i}(t) \, t^{i}$, and let~$s'_{i}$ be the
    minimum period of $C_{i}$ for ${0 \le i \le d+1}$.  Put $\fd F(t)
    \betodef \sum_{i=0}^{d} c_{i}(t)\, t^{i}$ and let $s_{i}$ be the
    minimum period of $c_{i}$ for $0 \le i \le d$.  Then, by
    \autoref{thm: sam and woods lemma},
    \begin{equation*}
        \mqp(F) 
        =
        \mqp(H) 
        =%
        \lcm\setof{s'_{0}, \dotsc, s'_{d}} 
        \divides%
        \lcm\setof{s_{0}, \dotsc, s_{d}} 
        =%
        \mqp(\fd F).
    \end{equation*}
    In particular, $\mqp(F) \le \mqp(\fd F)$.
\end{proof}

\subsection{%
    Proof that \texorpdfstring{$F_{n}$}{\uniF\unisubn} has minimum
    quasiperiod \texorpdfstring{$n$}{\unin}
}\label{subsec: minimum quasiperiod}%

In this subsection, we apply the result of \autoref{subsec: tools for
bounding} to show that $F_{n}$ has minimum quasiperiod $n$, thereby
completing the proof of \autoref{thm: main theorem}.  Thus far,
$F_{n}(k)$ has been defined only for nonnegative $k$.  Henceforth, we
also write~$F_{n}$ for the unique quasipolynomial extension of $F_{n}$
to all of $\Z$.

\begin{lemma}\label{lem:nthDifferenceOfFn}
    Let $n \ge i \in \Z_{\ge 0}$.  The $i$\textsuperscript{th}
    difference of $F_{n}$ satisfies
    \begin{equation} \label{eq:nthDifferenceOfFn-1}
        \fd^{i} F_{n} (t)
        =
        \sum_{\substack{j \in [0,t]_{\Z} \sst \\ j \equiv t \mod{n}}}
        \binom{j}{n-i}
        \qquad
        \text{for $t \in \Z_{\ge 0}$}.
    \end{equation}
    In particular,
    \begin{equation*}
        \fd^{n} F_{n} (t)
        =
        \floor{\frac{t}{n}} + 1
        \qquad
        \text{for $t \in \Z$},
    \end{equation*}
    and so $\fd^{n} F_{n}$ is a quasipolynomial with minimum
    quasiperiod~$n$.
\end{lemma}

\begin{proof}
    The claim is trivial when $i = 0$.  Proceeding by induction on
    $i$, we find that, for $n \ge i + 1$ and $t \in \Z_{\ge 0}$,
    \begingroup%
    \allowdisplaybreaks%
    \begin{align*}
        \fd^{i+1} F_{n} (t) 
        &=
        \fd^{i} F_{n} (t+1) - \fd^{i} F_{n} (t) \\
        &=%
        \sum_{\substack{j \in [0, t+1]_{\Z} \sst \\ j \equiv t + 1
        \mod{n}}} \binom{j}{n-i}
        -%
        \sum_{\substack{j \in [0, t]_{\Z} \sst \\ j \equiv t \mod{n}}}
        \binom{j}{n-i} \\
        &=%
        \sum_{\substack{j \in [-1, t]_{\Z} \sst \\ j \equiv t
        \mod{n}}} \binom{j + 1}{n-i}
        -%
        \sum_{\substack{j \in [0, t]_{\Z} \sst \\ j \equiv t \mod{n}}}
        \binom{j}{n-i} \\
        &=%
        \sum_{\substack{j \in [0, t]_{\Z} \sst \\ j \equiv t \mod{n}}}
        \binom{j}{n-(i+1)}.
    \end{align*}
    \endgroup%
    (The condition that $n \ge i + 1$ is used to eliminate the $j =
    -1$ term in the first sum on the right-hand side of the third
    equation.)  Thus, \cref{eq:nthDifferenceOfFn-1} is proved.  In
    particular, for $t \in \Z_{\ge 0}$,
    \begin{equation*}
        \fd^{n} F_{n} (t)
        =
        \#
        \setof{j \in [0,t]_{\Z} \sst j \equiv t \mod{n}}
        =
        \floor{\frac{t}{n}} + 1.
    \end{equation*}
    Therefore, $\fd^{n} F_{n} (t) = \floor{\frac{t}{n}} + 1$ for all
    $t \in \Z$.
\end{proof}

\begin{corollary}\label{cor: quasiperiod of Fn}
   The minimum quasiperiod of $F_{n}$ is $n$.
\end{corollary}

\begin{proof}
    By \autoref{lem:nthDifferenceOfFn}, $\fd^{n} F_{n}$ has minimum
    quasiperiod $n$.  Therefore, by \autoref{cor: difference has same
    min period}, $F_{n}$ itself has minimum quasiperiod $n$.
\end{proof}

Indeed, it follows from \autoref{thm: sam and woods lemma} that all
coefficient functions of $F_n$ are constant, except for the degree-$0$
coefficient function, which has minimum period $n$.  For, in the
notation of \autoref{thm: sam and woods lemma}, $f(t) \deftobe
\floor{\frac{t}{n}} + 1 = \frac{1}{n}t + c_{0}(t)$, where $c_{0}(t) =
-r_{n}(t)/n + 1$ and $r_{n}(t)$ is the remainder of $t$ modulo $n$.
Thus, $s_{0} = n$ and $s_{i} = 1$ for $i \ge 1$.  Any function $F$
such that $\fd F = f$ differs from $t \mapsto \sum_{x = 0}^{t-1} f(x)$
by a constant and so the minimum quasiperiod of the coefficient
function $C_{i}$ of $F$ divides $\lcm\setof{s_{i}, \dotsc, s_{d}} = 1$
for all $i$ such that $d \ge i \ge 1$.  That is, $C_{i}$ is
constant for all $i \ge 1$.  By induction, every coefficient function
of $F_{n}$, except for the degree-$0$ coefficient function, is
constant.

We are now prepared to prove our main result.

\begin{proof}[Proof of \autoref{thm: main theorem}]
    Let $G_n$ be the graph from \autoref{defn: G_n}.  The element
    $(\Lstar,n-1) \in \Phi(P_{G_n})$ from \autoref{defn: L labeling}
    is completely fundamental in $\Phi(P_{G_n})$, as shown in
    \autoref{thm: vertices of PGn}.  Hence $P_{G_n}$ has denominator
    $n-1$, as desired.

    By \autoref{prop: Mn to Fn-1}, the magic labeling quasipolynomial
    $M_n(k)$ is a sum of a polynomial and the quasipolynomial
    $F_{n-1}$.  We have by \autoref{cor: quasiperiod of Fn} that
    $F_{n-1}$ has minimum quasiperiod $n-1$.  Therefore $M_n(k)$ also
    has minimum quasiperiod $n-1$, thus proving the second statement.
\end{proof}

\section{%
    Graphs with Small Magic-Labeling Quasiperiods
}\label{sec: small quasiperiod}%

Though \autoref{thm: main theorem} demonstrates that the
quasipolynomial $M_G(k)$ can have arbitrarily large minimum
quasiperiod, there are still large families of graphs for which we can
give a uniform bound on the minimum quasiperiod.  In this section, we
show that for a large family of graphs, including any graph with a
leaf, the minimum quasiperiod of this quasipolynomial is at most $2$.

\begin{prop}\label{prop: common edge upper bound}
    Suppose there is an edge $e$ in the graph $G$ that attains the
    maximum label in every magic labeling of $G$ of index $2$.  Then
    $M_G(k)$ is a quasipolynomial of quasiperiod at most $2$.  If $G$
    is furthermore bipartite, then $M_G(k)$ is a polynomial.
\end{prop}

\begin{proof}
    Note that the statement holds trivially if the only magic labeling
    of~$G$ is the zero labeling, so we can suppose that $G$ admits a
    nonzero magic labeling.  Fix a magic labeling $L$ of $G$ of index
    greater than $2$.  By \autoref{lem: stanley cf index}, we can
    decompose $L$ as a sum of at least two magic labelings $L_i$ of
    $G$, each of index at most $2$.  Since $e$ attains the maximum
    label in each $L_i$, it must also attain the maximum label in $L$.
    We then have $\max(L) = \sum_{i} \max(L_i)$, and hence
    $(L,\max(L)) = \sum_{i} (L_i,\max(L_i))$ in $\Phi(P_G)$.  So we
    can conclude that $(L,\max(L))$ is not completely fundamental in
    $\Phi(P_G)$.  Therefore, any completely fundamental element of
    $\Phi(P_G)$ has index (and hence maximum label) at most $2$.  It
    follows from \autoref{cor: quasiperiod ub} that $2$ is a
    quasiperiod of $M_G(k)$.

    If $G$ is bipartite, we can furthermore decompose $L$ into a sum
    of magic labelings $L_i$ of index $1$ by \autoref{lem: stanley cf
    index}.  The same approach then applies to show that any
    completely fundamental labeling of $\Phi(P_G)$ has index $1$.
    Thus, \autoref{cor: quasiperiod ub} implies that $1$ is a
    quasiperiod of $M_G(k)$, i.e., $M_G(k)$ is a polynomial.
\end{proof}

\begin{cor}
    Let $G$ be a graph with a leaf.  Then $M_G(k)$ is a
    quasipolynomial of quasiperiod at most $2$.
\end{cor}

\begin{proof}
    The edge incident to the leaf vertex satisfies the hypotheses of
    \autoref{prop: common edge upper bound}, as its label in any magic
    labeling must always equal the index and hence is maximal.
\end{proof}

\begin{rem}
    While the hypothesis of \autoref{prop: common edge upper bound}
    only applies to magic labelings of index $2$, it is in fact
    equivalent to assert that the hypothesis holds for all magic
    labelings.  This follows from \autoref{lem: stanley cf index}, as
    any magic labeling can be decomposed as a sum of magic labelings
    of index at most $2$.  Moreover, any magic labeling of index $1$
    can be doubled to yield a magic labeling of index~$2$.
\end{rem}

\begin{rem}
    Let $G$ be a graph with an even number of vertices.  The
    \defing{matching preclusion number} $\mprec(G)$ of $G$ is the
    minimum cardinality of an edge set $S$ such that $G - S$ has no
    perfect matching.  If $G$ is bipartite, then the condition of
    \autoref{prop: common edge upper bound} is equivalent to the
    condition that $\mprec(G) \leq 1$.
\end{rem}

\begin{exmp}
    We can give a method for constructing a bipartite graph with
    matching preclusion number $1$.  For $i \in \{1,2\}$, let $G_i$ be
    a bipartite graph with a vertex $v_i$ such that $G_i \setminus
    \{v_i\}$ has a perfect matching.  Note that such a graph~$G_i$
    must have odd size.  Then consider the graph formed by connecting
    $G_1$ and $G_2$ by a single edge between $v_1$ and $v_2$.  The
    resulting graph has a perfect matching, since we can take the
    perfect matchings in $G_i \setminus \{v_i\}$ for $i \in \{1,2\}$
    along with the added edge.  However, the graph obtained by
    removing the edge between~$v_1$ and $v_2$ has no perfect
    matchings, since it consists of two components of odd size.
\end{exmp}

We have shown in this section that the quasiperiod of the
quasipolynomial $M_G(k)$ is small for certain families of graphs,
while we proved in \autoref{subsec: minimum quasiperiod} that the
quasiperiod is unbounded in general.  This leaves much room for future
progress in understanding how the quasipolynomial $M_G(k)$ behaves for
other types of graphs.

\addtocontents{toc}{\protect\setcounter{tocdepth}{0}}
\section*{Acknowledgements}
This work was completed in part at the 2022 Graduate Research Workshop
in Combinatorics, which was supported in part by NSF grant
$\#$1953985, and a generous award from the Combinatorics Foundation.
The second author was supported by the NSF GRFP and the Jack Kent
Cooke Foundation.  The authors thank Matt Beck and Maryam Farahmand
for many helpful discussions.  We also thank Jeremy Martin for early
contributions to the project.

\bibliography{refs} 

\addtocontents{toc}{\protect\setcounter{tocdepth}{2}}

%\nocite{*}
 \bibliographystyle{amsplain}

\end{document}

%% file: Figures/cone_example-latex.pdf_tex
%% Creator: Inkscape 1.3.2 (091e20e, 2023-11-25), www.inkscape.org
%% PDF/EPS/PS + LaTeX output extension by Johan Engelen, 2010
%% Accompanies image file 'cone_example-latex.pdf' (pdf, eps, ps)
%%
%% To include the image in your LaTeX document, write
%%   \input{<filename>.pdf_tex}
%%  instead of
%%   \includegraphics{<filename>.pdf}
%% To scale the image, write
%%   \def\svgwidth{<desired width>}
%%   \input{<filename>.pdf_tex}
%%  instead of
%%   \includegraphics[width=<desired width>]{<filename>.pdf}
%%
%% Images with a different path to the parent latex file can
%% be accessed with the `import' package (which may need to be
%% installed) using
%%   \usepackage{import}
%% in the preamble, and then including the image with
%%   \import{<path to file>}{<filename>.pdf_tex}
%% Alternatively, one can specify
%%   \graphicspath{{<path to file>/}}
%% 
%% For more information, please see info/svg-inkscape on CTAN:
%%   http://tug.ctan.org/tex-archive/info/svg-inkscape
%%
\begingroup%
  \makeatletter%
  \providecommand\color[2][]{%
    \errmessage{(Inkscape) Color is used for the text in Inkscape, but the package 'color.sty' is not loaded}%
    \renewcommand\color[2][]{}%
  }%
  \providecommand\transparent[1]{%
    \errmessage{(Inkscape) Transparency is used (non-zero) for the text in Inkscape, but the package 'transparent.sty' is not loaded}%
    \renewcommand\transparent[1]{}%
  }%
  \providecommand\rotatebox[2]{#2}%
  \newcommand*\fsize{\dimexpr\f@size pt\relax}%
  \newcommand*\lineheight[1]{\fontsize{\fsize}{#1\fsize}\selectfont}%
  \ifx\svgwidth\undefined%
    \setlength{\unitlength}{207.48675489bp}%
    \ifx\svgscale\undefined%
      \relax%
    \else%
      \setlength{\unitlength}{\unitlength * \real{\svgscale}}%
    \fi%
  \else%
    \setlength{\unitlength}{\svgwidth}%
  \fi%
  \global\let\svgwidth\undefined%
  \global\let\svgscale\undefined%
  \makeatother%
  \begin{picture}(1,0.95557169)%
    \lineheight{1}%
    \setlength\tabcolsep{0pt}%
    \put(0,0){\includegraphics[width=\unitlength,page=1]{cone_example-latex.pdf}}%
    \put(0.04094058,0.9221923){\color[rgb]{0,0,0}\makebox(0,0)[t]{\lineheight{1.25}\smash{\begin{tabular}[t]{c}\textbf{$z$}\end{tabular}}}}%
    \put(0.66938429,0.00764587){\color[rgb]{0,0,0}\makebox(0,0)[t]{\lineheight{1.25}\smash{\begin{tabular}[t]{c}\textbf{$x$}\end{tabular}}}}%
    \put(0.54642795,0.34839606){\color[rgb]{0,0,0}\makebox(0,0)[t]{\lineheight{1.25}\smash{\begin{tabular}[t]{c}\textbf{$y$}\end{tabular}}}}%
  \end{picture}%
\endgroup%